\documentclass{amsart}
\usepackage{amsmath}
\usepackage{amssymb}
\usepackage{amsthm}
\usepackage{tikz-cd}
\usepackage[style=numeric,sorting=nty]{biblatex}
\usepackage[english]{babel}

\title{Some Obstructions to Solvable Points on Higher Genus Curves}
\author{James Rawson}
\address{James Rawson \\ Mathematics Institute, University of Warwick, Coventry \\ United Kingdom}
\email{james.rawson@warwick.ac.uk}
\urladdr{https://warwick.ac.uk/fac/sci/maths/people/staff/rawson/}

\thanks{The author is supported by the Warwick Mathematics Institute Centre for Doctoral Training, and gratefully acknowledges funding from the UK Engineering and Physical Sciences Research Council (Grant number: EP/W523793/1)}

\newtheorem{theorem}{Theorem}
\newtheorem{corollary}{Corollary}
\newtheorem{lemma}{Lemma}
\newtheorem*{question}{Question}
\newtheorem*{conjecture}{Conjecture}
\newtheorem{defn}{Definition}
\newtheorem{propn}{Proposition}

\begin{document}
\begin{abstract}
It is known that for a curve defined over $\mathbb{Q}$ of genus $g \leq 4$, there exists a point on the curve defined over a solvable extension of $\mathbb{Q}$. We relate points on curves of genus $g \geq 5$ over solvable extensions to the Bombieri-Lang conjecture. Specifically, we show that varieties parametrising points defined over extensions with a fixed solvable Galois group are of general type. Moreover, we show the existence of certain subvarieties in these varieties imply the existence of solvable morphisms from the curve.
\end{abstract}
\maketitle
\section{Introduction}
For curves of genus at least 2, Faltings' theorem states there are only finitely many points defined over any (fixed) number field. On the other hand, points defined over the algebraic closure of $\mathbb{Q}$ are Zariski dense. Therefore, it is natural to ask about points defined over some class of number fields. One such class of classical interest is the solvable number fields, prompting the following.
\begin{question}
Given a curve $C$ defined over a number field $K$, does there exist a number field $F / K$, with solvable Galois group, such that $C(F) \neq \emptyset$?
\end{question}
P\'al proved that for any geometrically irreducible curve of genus 0, 2, 3, or 4 defined over a perfect field, there is at least one solvable point \cite{pal}, and Wiles \& \c{C}iperani proved the same for genus 1 curves over $\mathbb{Q}$, provided there are points over every completion \cite{wiles}. It is true for all $C$, when $K$ is instead a finite extension of $\mathbb{Q}_p$, since there are $\overline{\mathbb{Q}_p}$ points, and all finite extensions of $\mathbb{Q}_p$ are solvable. P\'al \cite{pal} constructed counterexamples for almost all $g \geq 5$ over local fields where the absolute Galois group of the residue field has quotients isomorphic to $S_5$, $\mathrm{PSL}_3(\mathbb{F}_2)$ and $\mathrm{PSL}_3(\mathbb{F}_3)$. The question remains open for curves defined over number fields, as local fields of the type in P\'al's construction do not arise as completions of a number field. We will show there are ``not too many'' points for a curve of genus $g \geq 5$ defined over number fields having a given solvable Galois group, assuming the following conjecture of Bombieri and, independently, Lang \cite{lang}.
\begin{conjecture}[Bombieri-Lang]
Let $V$ be a variety of general type, then rational points on $V$ are not Zariski dense.
\end{conjecture}

We first recall the definition of general type:
\begin{defn}
A smooth variety is of general type if its Kodaira dimension is equal to the dimension of the variety
\end{defn}
Unwinding the definition, gives the following.
\begin{defn}
Let $V$ be an $n$-dimensional projective variety, then $V$ is of general type if for large enough $r$,  the image of the map associated to $rK_V$ (where $K_V$ is the canonical divisor) is $n$-dimensional.
\end{defn}
The associated map in the above definition is the map (unique up to isomorphisms of $\mathbb{P}^k$) to $\mathbb{P}^k$ where each coordinate is an element of a basis of $\mathcal{L}(r K_V) = \{f \in k(V) | \mathrm{Div}(f) + r K_V \geq 0\}$. For example, a curve is of general type if and only if it has genus $g \geq 2$. In this case, the Bombieri-Lang conjecture is simply Faltings' theorem. 
In this paper, it will be necessary to study singular varieties and so a notion of general type for such varieties is needed. It is known that Kodaira dimension is a birational invariant, and so the definition of general type can be extended as follows.
\begin{defn}
A variety is of general type if a desingularisation is of general type.
\end{defn}

We return to the problem of points with a given Galois group. We will fix a number field, $K$, and assume the curve, $C$, is smooth, projective, and geometrically irreducible. 

Fix a transitive subgroup $G \leq S_n$, then $G$ acts on $C^n$ by permuting the factors. As $G$ is finite and its orbits are contained in affine patches, the quotient variety, $C^n / G$, exists. Points on the variety will be denoted as $[(P_1, ..., P_n)]$, where $(P_1, ..., P_n)$ is a representative of the equivalence class. Rational points on such varieties are closely related to points on the original curve with given Galois group. More precisely:
\begin{propn}
A rational point on $C^n / G$ is a union of Galois orbits of points on $C$ defined over fields with Galois group contained in $G$. 
\end{propn}
\begin{proof}
Let $[(P_1, ..., P_n)]$ be a rational point on $C^n / G$, then for every $\sigma \in \mathrm{Gal}(\overline{K} / K)$, $\sigma [(P_1, ..., P_n)] = [(P_1, ..., P_n)]$. The Galois action is coordinate-wise, and so $[(P_1, ..., P_n)] = [(\sigma P_1, ..., \sigma P_n)]$. Therefore $\sigma P_i$ is $P_j$ for some $j$. Let $P_{i_1}, ..., P_{i_m}$ be a Galois orbit, then the homomorphism $$\mathrm{Gal}(K(P_{i_1}, ..., P_{i_m}) / K) \to S_n$$ must have image contained in $G$. Repeating this for each Galois orbit gives the desired statement.
\end{proof}

This motivates the study of these varieties: to understand points on the curve with given Galois group. The rest of the paper is dedicated to the structure of these varieties. To this end, we determine when these varieties are of general type.
\begin{theorem}
Suppose $G$ contains exactly $m$ transpositions of the form $(1, i)$. Then $C^n / G$ is of general type if and only if $g(C) > m + 1$.
\label{mainthm}
\end{theorem}
This is proved in Section 2.

\begin{corollary}
If $G$ is a solvable group and $g(C) \geq 5$, then $C^n / G$ is of general type.
\end{corollary}
\begin{proof}
Suppose $G$ contains exactly $m$ transpositions of the form $(1, i)$. Then $G$ contains a subgroup isomorphic to $S_{m + 1}$. As $G$ is solvable, $m + 1 \leq 4$. Then $g(C) > m + 1$, and so by Theorem 1, $C^n / G$ is of general type.
\end{proof}
Combining the corollary with the Bombieri-Lang conjecture, we get the following theorem.

\begin{theorem}
Assume the Bombieri-Lang conjecture. Let $G$ be a solvable group, and $C$ be a smooth curve of genus at least 5. Then for any transitive embedding $G \hookrightarrow S_n$, the rational points on $C^n / G$ are not Zariski dense. 
\end{theorem}

For curves of genus 2, 3 and 4, we can easily compute the Kodaira dimension of $C^n / G$, where $G$ is the Galois group of points defined by P\'al's method. 

For $g(C) = 2$, $C$ is hyperelliptic, and so pulling rational points of $\mathbb{P}^1$ back through the double covering generically gives points defined over number fields with Galois group $S_2$. The variety $C^2 / S_2$ is birational to the Jacobian, and has Kodaira dimension 0.

When $g(C) = 3$, $C$ is generically a plane quartic. Intersecting the quartic with lines in $\mathbb{P}^2$ gives points with Galois groups $S_4$ generically. The quotient $C^4 / S_4$ is uniruled over the Jacobian, and so has Kodaira dimension $-\infty$. 

For the $g(C) = 4$ case, $C$ is generically the intersection of a quadric surface and a cubic surface. P\'al starts by taking a point on the quadric (generically defined over a quadratic extension), and finding a ruling through the point (defined over a further quadratic extension). Intersecting this line with the cubic surface gives a point of the curve defined generically over a further $S_3$ extension. The Galois group of this point is then generically the subgroup of $S_{12}$ generated by $(1, 2)$, $(1, 3)$, $(1, 4, 7, 10)$ and $(1, 4)(7, 10)$. There are exactly 2 transpositions in the group that act non-trivially on 1, and $4 > 3$, so $C^{12} / G$ is of general type --- its Kodaira dimension is 12. 

The latter example shows it is not enough for $C^n / G$ to be of general type. In the $g(C) = 4$ case, the solvable points arise from a low degree morphism from the curve to $\mathbb{P}^1$. For higher genus curves, there will not be low degree morphisms, and so we do not expect there to be many solvable points. We prove that if there are rational points distributed, in a precise sense, like those arising from a morphism, then those points do indeed arise from a morphism. 

\begin{defn}
Fix a curve $C$. A rational curve $D$ in $\mathrm{Sym }^n C$ is of fibre type if $D$ generically intersects the image of $\{P\} \times C^{n - 1}$ in a single point (as $P$ varies over $C$). A curve in $C^n / G$ is of fibre type if it maps injectively onto its image in $\mathrm{Sym }^n C$ and its image is of fibre type.
\end{defn}

We prove the following theorem about curves of fibre type in Section 3.
\begin{theorem}
Let $G$ be a transitive subgroup of $S_n$, and suppose there exists a curve of fibre type in $C^n / G$ with Zariski dense rational points, then $C$ has a morphism to $\mathbb{P}^1$ with Galois group $G$. 
\label{ratcurve}
\end{theorem}

Combining this with the geometry of generic high genus curves gives the following, which gives further control on the geometry of the varieties $C^n / G$. 
\begin{corollary}
Let $C$ be a very general curve of genus $\geq 7$, and $G$ a solvable, transitive subgroup of $S_n$ for some $n$. Then $C^n / G$ is a variety of general type, and contains no curves of fibre type. 
\end{corollary}
\begin{proof}
The statement about being of general type follows from the previous corollary.

A theorem of Zariski \cite{zariski} states that for a very general curve of genus $\geq 7$, the Galois group of any morphism $C \to \mathbb{P}^1$ is not solvable. By Theorem 3, the existence of a curve of fibre type (defined over an extension $L / K$) with dense $L$-rational points, would imply the existence of a morphism from $C$ with Galois group $G$, contradicting Zariski's theorem. Thus, there are no curves of fibre type with dense $L$-rational points, for any field extension $L / K$. As any curve of fibre type is rational, there exists a field extension where its rational points are dense, thus there can be no such curves.
\end{proof}

It is unclear if there exist curves defined over $K$ such that the hypothesis is satisfied, as very general entails the removal of countably many subvarieties of the moduli space, and the algebraic closure of $K$ is also countable. 
The fibre type condition is quite restrictive, but for $n$ near the gonality of $C$, the dimensions of the fibres of $\mathrm{Sym}^n C$ over $\mathrm{Jac}(C)$ are small enough that this behaviour is typical in the symmetric power.

\section{When Quotients are of General Type}
\subsection{The no transpositions case}
First, we prove the following special case,
\begin{theorem}
Let $C$ be a smooth curve of genus at least 2, and let $G \subset S_n$ be a subgroup that contains no transpositions. Then $C^n / G$ is of general type.
\label{notrans}
\end{theorem}

We first observe that as $G$ is a subgroup of $S_n$, there is a map $\pi_2 : C^n / G \to \mathrm{Sym}^n C$ so that $\pi_1 : C^n \to \mathrm{Sym}^n C$ factors through $C^n \to C^n / G$. The variety $C^n / G$ may be singular, and so it needs to be replaced by a desingularisation, $X$, which has a birational morphism $X \to C^n / G$. Composing this map with $\pi_2$ gives a map $\phi : X \to \mathrm{Sym}^n C$. Although $\pi_2$ is finite, $\phi$ may not be, as the map $X \to C^n / G$ may contract the pre-image of the singular locus, but it is generically finite. This is summarised in the following commutative diagram.

$$\begin{tikzcd}
    C^n \arrow[dd, "\pi_1"] \arrow[dr] & &\\
    & C^n / G \arrow[dl, "\pi_2"] & X \arrow[l] \arrow[dll, "\phi", bend left] \\
    \mathrm{Sym}^n C & &
  \end{tikzcd}$$

The variety $C^n / G$ can only be singular where a point of $C^n$ is fixed by some non-trivial $g \in G$, as elsewhere the quotient map $C^n \to C^n / G$ is finite \'{e}tale. As $G$ contains no transpositions, if $(P_1, ..., P_n)$ is fixed by a non-trivial element of $G$, either at least two pairs of $P_i$ are equal, or at least three $P_i$ are equal. Therefore the singular locus is contained in a subvariety of codimension at least $2$.

To prove the theorem, some results are needed about the structure of the maps $\pi_i$.

\begin{lemma}
The ramification locus of $\pi_1$ is the set $\Delta = \{(P_1, ..., P_n) | P_i = P_j \text{ for some } i \neq j\}$, and under the hypotheses of the theorem, the ramification locus of $\pi_2$, away from the singular points of $C^n / G$, is the image of $\Delta$ under the quotient map. The ramification index of both maps along the respective divisors is 2.
\end{lemma}
\begin{proof}
We start with $\pi_1$. The pre-image, $\pi_1^{-1}(P_1 + ... + P_n)$, generically consists of $n!$ points, $(P_{\sigma(1)}, ..., P_{\sigma(n)})$ as $\sigma$ ranges over $S_n$.  There is, therefore, ramification precisely when two of these images collide, when $P_i = P_j$ for some $i \neq j$. 

As $\pi_1$ factors through $C^n / G$, the map $\pi_2$ is unramified outside of the image of $\Delta$. The ramification index of $\pi_1$ at $\pi_1(R)$, for some $R \in \Delta$, is equal to the size of its stabiliser in $S_n$. For a generic element of $\Delta$, this is 2 (as the stabiliser is just the transposition switching the identical entries). Similarly, the ramification index for the map to $C^n / G$ at $R$ is the size of its stabiliser in $G$. Away from the singular locus, this stabiliser is trivial. Therefore, as ramification indices are multiplicative under composition, the ramification index of $\pi_2$ at the image of $R$ must be equal to that of $\pi_1$ at $R$. In particular, the map is ramified there.
\end{proof}

And now to prove Theorem~\ref{notrans},
\begin{proof}
As being of general type (for smooth varieties) is a statement about canonical divisors, the first step is to relate $K_{C^n}$ and $K_X$. We will obtain such a relation by considering the two maps to $\mathrm{Sym}^n C$. 

There is a generalisation of the Riemann-Hurwitz formula for a generically finite morphism of smooth varieties $f : X \to Y$, which states $K_X \simeq f^* K_Y + R + E$ where $R$ is an effective divisor supported on the ramification locus, and $E$ is an effective divisor supported on the exceptional locus \cite{Debarre}. Applying this to $\pi_1$ gives the following, as the ramification index is 2 on the divisor $\Delta$.
$$K_{C^n} \simeq \pi_1^* K_{\mathrm{Sym}^n C} + \Delta$$
This can be pushed forward to yield:
$$n! K_{\mathrm{Sym}^n C} \simeq {\pi_1}_* K_{C^n} - {\pi_1}_* \Delta$$

By the preceding lemma, $\pi_2$ is ramified along the image of $\Delta$, except for the codimension 2 or smaller set where $C^n / G$ is possibly singular. Therefore $\phi$ must also be ramified along the image of $\Delta$, again with ramification index 2. Applying the same formula to $\phi$, and denoting the image of $\Delta$ in $X$ as $\Delta'$, and the ramification and exceptional locus of $\phi$ as $\Delta' + R'$ and $E'$ respectively.
$$K_X \simeq \phi^* K_{\mathrm{Sym}^n C} + \Delta' + R' + E'$$

Combining both formulae:
$$n! K_X \geq \phi^* {\pi_1}_* K_{C^n} - \phi^* {\pi_1}_* \Delta + n! \Delta' + n! E'$$
The divisor ${\pi_1}_* \Delta$ is the image of $\Delta$ under $\pi_1$, counted with multiplicity $\frac{n!}{2}$, as this is the degree of $\pi_1$ restricted to $\Delta$. Similarly, $\phi^* {\pi_1}_* \Delta$ is supported on the sum of $\Delta'$ and some exceptional divisors. The component supported on $\Delta'$ has multiplicity $n!$, as $\phi$ is ramified with index 2 along $\Delta'$. The exceptional components will be contained within $n! E$, therefore $n! K_X \geq \phi^* {\pi_1}_* K_{C^n}$. 
Pullbacks only increase the number of global sections, thus to show $C^n / G$ is of general type it is enough to check that ${\pi_1}_* K_{C^n}$ has $n$ algebraically independent sections.

As $C$ is of general type, $K_C$ has a non-trivial section $f$. Define $f_i$ on $C^n$ by $f_i(P_1, ..., P_n) = f(P_i)$, this is a section of $K_{C^n}$. Taking the elementary symmetric polynomials in $f_i$ give sections of $K_{C^n}$ which are $G$-invariant, and are algebraically independent. These are therefore sections of ${\pi_1}_* K_{C^n}$, and the result follows.
\end{proof}

This establishes the main theorem in the case $m = 0$. 

\subsection{Proof of Theorem~\ref{mainthm}}
The general case will be proven by factoring the quotient map through symmetric powers, before using a result on symmetric powers of higher dimensional varieties to conclude. Before this, we need some lemmas on the structure of transitive subgroups of $S_n$. The first of which is a slight generalisation of a result in Clark's ``Elements of Abstract Algebra'' \cite{clark}.

\begin{lemma}
Let $G \subset S_n$ be a transitive subgroup, and suppose it contains exactly $m$ transpositions of the form $(1, i)$, then $m + 1 \mid n$.
\end{lemma}
\begin{proof}
Define an equivalence relation on $\{1, ..., n\}$ by $i \sim j$ if $i = j$ or $(i, j) \in G$. Reflexivity and symmetry are obvious from the definition. Suppose $i \sim j$ and $j \sim k$, then $(i, j), (j, k) \in G$. As $G$ is a subgroup, $(i, k) = (j, k)(i, j)(j, k) \in G$, and so $i \sim k$. Therefore $\sim$ is transitive, and it defines an equivalence relation. 

Consider an equivalence class $[i]$. By transitivity of $G$, there is a $\sigma \in G$ such that $\sigma(1) = i$. If $j \in [i]$, $(i, j) \in G$, and so $(1, \sigma^{-1}(j)) = \sigma^{-1} (i, j) \sigma \in G$, and $\sigma^{-1}(j) \in [1]$. Therefore $|[i]| \leq |[1]|$, and by symmetry, the inequality holds in the opposite direction, so all equivalence classes have the same size. As $[1]$ contains 1 along with an element for each transposition containing 1, $|[1]| = m + 1$. Partitioning $\{1, ..., n\}$ into equivalency classes gives the divisibility.
\end{proof}

The second lemma describes a quotient of transitive subgroups.
\begin{lemma}
Let $G \subset S_n$ and $m$ be as before, let $n' = \frac{n}{m + 1}$ and define $H$ to be the subgroup of $G$ generated by all transpositions, then $H$ is a normal subgroup, $H \cong S_{m + 1} \times \ldots \times S_{m + 1}$ and there is a natural inclusion $G / H \hookrightarrow S_{n'}$.
\label{normal}
\end{lemma}
\begin{proof}
Group the transpositions according to the equivalence classes of the previous lemma. Transpositions in the same class will generate a symmetric group on the size of the class, $m + 1$, and transpositions from different classes commute. This gives $H \cong S_{m + 1} \times \ldots \times S_{m + 1}$, with the product running over equivalence classes.

Consider the transpositions corresponding to the first equivalence class; these generate a subgroup $H'$. The conjugates of $H'$ are the groups generated by the other equivalence classes. Let $g \in G$ be an element fixing all the conjugates of $H'$ under conjugation. As $g (i, j) g^{-1} = (g(i), g(j))$, this shows $g(i) \sim i$ for all $i$, therefore $g \in H$. Conversely, all elements of $H$ preserve $H'$ and its conjugates. The homomorphism $G \to S_{n'}$ induced by the conjugation action, therefore, has kernel $H$. 
\end{proof}

One final result is needed, describing the geometry of symmetric powers. This result does not seem to appear in the literature, although similar results do. We therefore give it here, along with the proof provided by MathOverflow user, Olivier Benoist \cite{mathoverflow}.
\begin{lemma}
Let $C$ be a curve of genus $g$. The variety $\mathrm{Sym}^n C$ is of general type for $n < g$, (birational to) an abelian variety for $n = g$ and uniruled for $n > g$.
\end{lemma}
\begin{proof}
The statements for $n \geq g$ follow from Riemann-Roch and the definition of the Jacobian variety, so it remains to prove it for $n < g$.

The Abel-Jacobi map gives a birational map $\mathrm{Sym}^n C \to \mathrm{Jac}(C)$, and so it is enough to prove the image, $W_n$, is of general type. If it were not, by a theorem of Ueno \cite{ueno}, it would contain an abelian variety $A$ such that $A + W_n = W_n$. In particular, as $W_{g - 1}$ is expressible as the sum of elements from $W_n$, $W_{g - 1}$ is invariant under addition by $A$. For any $x \notin W_{g - 1}$, the locus $A + x$ is positive dimensional and disjoint from $W_{g - 1}$,  but $W_{g - 1}$ is an ample divisor and so must intersect any positive dimensional subvariety.
\end{proof}

Finally, a proof of the main theorem
\begin{proof}[Proof of Theorem~\ref{mainthm}]
Let $G$ be a transitive subgroup of $S_n$. If $m = 0$, then the previous section gives the desired result, otherwise, $m > 0$, and $H$ (as defined in lemma~\ref{normal}) is non-trivial. The map $C^n \to C^n / G$ factors through $C^n / H$, and this quotient is $\left(\mathrm{Sym}^{m + 1} C\right)^{n'}$. Let $V = \mathrm{Sym}^{m + 1} C$, then $V$ is of general type if and only if $g > m + 1$ by the preceding lemma. It remains to understand $V^{n'} / (G / H)$, but this maps surjectively onto $\mathrm{Sym}^{n'} V$ through a finite map. The symmetric power of a variety of dimension greater than 1 is of general type if and only if the variety is of general type \cite{arapura}. If $g > m + 1$, then by pulling back the canonical divisor to (a desingularisation of) $\mathrm{Sym}^{n'} V$ to $V^{n'} / (G / H)$, shows $V^{n'} / (G / H)$ is of general type. If $g \leq m + 1$, then $V$ is not of general type, so $V^{n'}$ is not, and so $V^{n'} / (G / H)$ is not either. 
\end{proof}

\section{Rational Curves on $C^n / G$}
We start by considering rational curves in $\mathrm{Sym}^n C$. 
\begin{lemma}
Suppose $\mathrm{Sym}^n C$ contains a curve of fibre type, then $C$ has a morphism, $f$, of degree at most $n$ to $\mathbb{P}^1$. Moreover, the fibres of this map as divisors are the points of the rational curve, up to fixed points
\end{lemma}
\begin{proof}
Let $D$ be the rational curve. There is a morphism $C \times \mathrm{Sym}^{n - 1} C \to \mathrm{Sym}^n C$ given by symmetrisation, and let $D'$ be the pre-image of $D$ under this map. Projection onto the first factor of the product gives a morphism $D' \to C$. For a generic $P \in C$, there is a unique point of $D'$ above $P$, as $D$ is of fibre type, therefore $D'$ contains an irreducible component, $C'$, isomorphic to $C$. There is a birational map $D \dashrightarrow \mathbb{P}^1$, and composing with the symmetrisation map gives $C' \to D' \to D \dashrightarrow \mathbb{P}^1$. This extends by completeness to a morphism $f : C = C' \to \mathbb{P}^1$.

As the symmetrisation map has degree $n$, $f$ has degree at most $n$. The fibres of $f$ are as stated, since points of $D'$ are the points of $D$ where one of the points in the support of the divisor is distinguished, and fixed points are removed as they belong to a separate component of $D'$. 
\end{proof}

From this construction it is also clear that the field of definition of $f$ is the same as the field of definition of the rational curve. 

We can now prove Theorem~\ref{ratcurve}, that curves of fibre type in $C^n / G$ imply the existence of rational maps on $C$ with Galois group $G$.
\begin{proof}[Proof of Theorem~\ref{ratcurve}]
As $C^n / G$ contains a curve of fibre type with Zariski dense rational points, $C$ has a rational map, $f$, to $\mathbb{P}^1$. The fibres of $f$ over rational points are rational points on $C^n / G$, in particular, the Galois group of any fibre of $f$ above a rational point is contained in $G$. By Hilbert's Irreducibility Theorem for Galois covers \cite{serre}, the Galois group of $f$ is $G$.
\end{proof}

We illustrate the utility of this result with the following
\begin{theorem}
There are at most finitely many cyclic number fields, $L$, of degree 3 over $\mathbb{Q}$, such that $X_0(34)(L) \neq X_0(34)(\mathbb{Q})$, where $X_0(34)$ denotes the modular curve of level $\Gamma_0(34)$.
\end{theorem}
\begin{proof}
Ozman and Siksek list $X_0(34)$ as a non-hyperelliptic genus 3 curve, where the Jacobian has rank 0 \cite{ozman}. 

There is a sequence of maps $X_0(34)^3 / C_3 \to \mathrm{Sym }^3 X_0(34) \to J_0(34)$. There are only finitely many rational points on the right hand side, therefore the rational points on $\mathrm{Sym }^3 X_0(34)$ are contained in a finite collection of rational curves, except for finitely many exceptions. These curves are necessarily of fibre type, since $X_0(34)$ is not hyperelliptic and so the Riemann-Roch space of a degree 3 divisor is at most 2 dimensional. Pulling back each rational curve to $X_0(34)^3 / C_3$ give one of two cases, a pair of rational curves or an irreducible double cover of $\mathbb{P}^1$.

In the former case, each must map injectively onto their image, and would be of fibre type. This would imply $X_0(34)$ has a map to $\mathbb{P}^1$ with Galois group $C_3$. In particular, $X_0(34)$ would have an automorphism of order 3, but the automorphism group is a product of cyclic groups of order 2 \cite{ogg}.

Therefore the pre-image of any rational curve in the symmetric power is a double cover. The double cover ramifies over points of the form $2P + Q$, and these are smooth points unless they are of the form $3P$. Suppose the double cover was not smooth in at least 2 points, then there are divisors $3P$ and $3Q$ in the same curve on the symmetric power, and so are linearly equivalent. As $P \neq Q$, $3P$ has a non-trivial section, and so by Riemann-Roch, $K_{X_0(34)} - 3P$ is effective. Therefore $K_{X_0(34)} \sim 3P + R$, and similarly, $K_{X_0(34)} \sim 3Q + S$ for some points $R, S \in X_0(34)$. Comparing these, shows $R \sim S$, and so $R = S$. For the canonical embedding, the canonical divisor is a hyperplane section, so $P$ and $Q$ are flexes, and their tangents meet at $R$.

Using an explicit model for $X_0(34)$ over the rationals, as $\mathbb{V}(F(X, Y, Z)) \subset \mathbb{P}^2$, the flex points can be computed as the vanishing locus of the Hessian determinant of $F$ along the curve. As the intersection of a quartic curve and a sextic, there are 24 such points. Using Groebner bases, these intersection points can be computed, and their coordinates are defined by a degree 24 irreducible polynomial. By working in a number field over which one of the flex points is defined, the tangent line to that point can be computed. As the tangent line meets the flex point with multiplicity 3, the line meets the curve at a fourth point which must be defined over the same field. The coordinates of this point each satisfy an irreducible degree 24 polynomial, and so each conjugate corresponds to exactly one flex, and no two flex lines meet at the same point on the curve. 

The pullback of a rational curve on $\mathrm{Sym }^3 X_0(34)$ therefore has at most 1 non-smooth ramification point. A triple cover of $\mathbb{P}^1$ by a genus 3 curve has total ramification degree 10, and so the pullback of a rational curve in the symmetric cube ramifies over its image at either 8 or 10 smooth points, corresponding to geometric genus 3 or 4 respectively. There can be at most finitely many rational points on such curves, and so there are finitely many points on $X_0(34)^3 / C_3$. 
\end{proof}

\subsection*{Acknowledgements}
The author would like to thank Samir Siksek for suggesting the question at the heart of the paper, as well as for many productive conversations, and the anonymous referee whose detailed feedback has helped improve this paper. 
\printbibliography
\end{document}